\documentclass[11pt,reqno]{amsart}
\usepackage{amssymb,amsmath,amsthm,epsfig,times}
\usepackage{mathrsfs}
\usepackage{epstopdf}
\usepackage{pdfpages}
\numberwithin{equation}{section} 
\usepackage{verbatim}
\usepackage[pdftex]{hyperref}

\def\C{{\mathbb C}}

\def\N{{\mathbb N}}
\def\P{{\mathbb P}}

\def\R{{\mathbb R}}
\def\*S{{\mathbb S}}

\def\interior{\hspace{0.02in}
\begin{picture}(6,6)
\put (0,0){\line(1,0){6}}
\put (6,0){\line(0,1){6}}
\end{picture} \hspace{0.04in}}
\newcommand{\dis}{\displaystyle}
\newcommand{\cal}{\mathcal}

\newtheorem{theorem}{Theorem}
\newtheorem{proposition}{Proposition}[section]
\newtheorem{lemma}[proposition]{Lemma}
\newtheorem{corollary}[proposition]{Corollary}

\begin{document}
\title[Explicit Hodge decomposition on Riemann surfaces]
{Explicit Hodge decomposition on Riemann surfaces}

\author{Gennadi M. Henkin}
\address{Institut de Mathematiques \\ Universite Pierre et Marie Curie \\ 75252 BC247 Paris \\
Cedex 05 \\ France, and CEMI  Acad. Sc. \\ 117418 \\ Moscow, Russia}
\curraddr{}
\email{guennadi.henkin@imj-prg.fr}

\author{Peter L. Polyakov}
\address{University of Wyoming \\ Department of Mathematics \\ 1000 E University Ave
\\ Laramie, WY 82071}
\curraddr{}
\email{polyakov@uwyo.edu}
\thanks{Second author was partially supported by the NEUP program of
the Department of Energy}



\begin{abstract}
We present a construction of an explicit Hodge decomposition for $\bar\partial$-operator
on Riemann surfaces.
\end{abstract}

\maketitle


\section{Introduction}\label{Introduction}

\indent
Classical Hodge decomposition on the space $Z^{(0,1)}(V)\subset {\cal E}^{(0,1)}(V)$
of smooth $\bar\partial$-closed $(0,1)$-forms on a smooth algebraic curve $V\subset\C\P^n$,
$(n\geq 2)$ with metric induced by Fubini-Study metric of $\C\P^n$, has the following form

\newtheorem{Hthm}{Hodge Theorem}
\renewcommand{\theHthm}{}
\begin{Hthm}\label{ClassicalHodge}(\cite{Ho,Wy1,Wy2,Kd1})
For any form $\phi\in Z^{(0,1)}\left(V\right)$ there exists a unique Hodge decomposition:
\begin{equation}\label{ClassicalDecomposition}
\phi=\bar\partial R_1[\phi]+ H_1[\phi],
\end{equation}
where $H_1$ is the orthogonal projection operator from $Z^{(0,1)}\left(V\right)$
onto the subspace ${\cal H}^{(0,1)}\left(V\right)$ of antiholomorphic $(0,1)$-forms on $V$,
$R_1=\bar\partial^*G_1$, $\bar\partial^*:{\cal E}^{(0,1)}(V)\to {\cal E}^{(0,0)}(V)$,
where $\bar\partial^*=-*\bar\partial*$ is the Hodge dual operator for $\bar\partial$,
$*$ is the Hodge operator, and $G_1$ is the Hodge-Green operator for Laplacian
$\triangle=\bar\partial\bar\partial^*+\bar\partial^*\bar\partial$ on $V$.
\end{Hthm}

\indent
Hodge Theorem was proved by Hodge in \cite{Ho} using the Fredholm's theory of integral
equations. Weyl used his method of orthogonal projection from \cite{Wy1} to correct and
simplify the Hodge's proof in \cite{Wy2}, and was followed by Kodaira in \cite{Kd1},
who also used Weyl's method of orthogonal projection.
However, the Hodge Theorem, as it is formulated above and in \cite{Dl, BDIP},
is not explicit enough for some applications. This disadvantage was pointed out by
Griffiths and Harris in (\cite{GH} \S 0.6), where the authors remarked that \lq\lq the
Hilbert space method has the disadvantage of not giving us the Green's operator\rq\rq\
in the form of an integral operator with \lq\lq a beautiful kernel on $M\times M$ with
certain singularities along the diagonal\rq\rq.\\
\indent
A specific problem that we have in mind is an explicit solution
of the inverse conductivity problem on a bordered Riemann surface, in which the conductivity function has to be reconstructed from the Dirichlet-to-Neumann map on its boundary
(see \cite{C}, \cite{HN}), in more general setting going back to \cite{Ge}.
We notice that article \cite{HN} of Henkin and Novikov on this subject was motivated
by article \cite {HP1} by the authors of the present article.
In \cite{HP4} we made the first step toward explicit solution of the inverse conductivity problem
by constructing an explicit Hodge-type decomposition for {\it $\bar\partial$-closed residual
currents} of homogeneity zero on reduced complete intersections in $\C\P^n$.
In the present article using Theorem 1 from \cite{HP4} we construct an explicit formula
for operator $R_1$ in \eqref{ClassicalDecomposition} assuming the knowledge of operator
$H_1$. A problem that is definitely worth considering is the construction of
an explicit form of $H_1$. Our choice of Riemann surfaces is motivated by the application
mentioned above, though we consider the generalization of Theorem 1 from \cite{HP4}
to locally complete intersections as another interesting and important task.\\
\indent
The main result of the present article is the construction in
Theorem~\ref{ExplicitHodge} of an explicit Hodge decomposition for $\bar\partial$-closed
forms on an arbitrary Riemann surface assuming the knowledge of the Hodge projection.
This construction is based on a generalized version of Theorem 1 from \cite{HP4}, which is
presented in section \ref{NewVersion} and covers the case of arbitrary homogeneity.
The decompositions in Theorem~\ref{HomogeneousHodge} and in the propositions below
are explicit in the sense that they depend only on the equations from \eqref{Variety}
describing $V$ as a subvariety of $\C\P^n$, and are defined by explicit integral operators
with singular kernels of the Coleff-Herrera \cite{CH} and of the Cauchy-Weil-Leray
types \cite{Wi, L}.\\
\indent
Construction of integral formulas on $\C\P^n$ with application to complex Radon transform
was initiated in \cite{HP1} and \cite{Bn}. Application of such formulas to solution of
$\bar\partial$-equation on singular analytic spaces was initiated in \cite{HP2}
and motivated further work in this direction (see for example \cite{AS, FG, Sa}).
The development of specific residual formulas in \cite{HP4} and
in Theorem~\ref{HomogeneousHodge} above has a long history going back to
Poincare \cite{P}, Leray \cite{L}, Grothendieck \cite{Gr}, Herrera-Lieberman \cite{HL},
Dolbeault \cite{Do}, and Coleff and Herrera \cite{CH}. As it is pointed out in \cite{CH},
the authors' work was conceived on the one hand as a generalization of the theory
of residues of meromorphic forms and of the Grothendieck's theory
of residues presented by Hartshorne \cite{Ha1}, and on the other hand of the work
of Ramis and Ruget \cite{RR} on dualizing complex in analytic geometry.\\
\indent
The modern development of the formulas of Cauchy-Weil-Leray
type was initiated in \cite{Kp}, \cite{Po}, \cite{Li}, \cite{O}.\\
\indent
As a preliminary step in the construction of the explicit Hodge decomposition
of Theorem~\ref{ExplicitHodge} we use Theorem~\ref{HomogeneousHodge}
and construct in Proposition~\ref{CurveHodge} an intermediate explicit Hodge-type
decomposition on an arbitrary compact Riemann surface ${\cal X}$, not necessarily embeddable
into $\C\P^2$. According to a classical result (see \cite{Kd2}), going back to Gauss
and Riemann, such surface admits an embedding as an algebraic submanifold in $\C\P^3$.
This embedding can be composed with a generic projection on $\C\P^2$ to produce
an immersion into $\C\P^2$ with only nodes as singularities (see \cite{GH, Ha2}). Then we use
the Hodge-type decomposition of Theorem~\ref{HomogeneousHodge} on the
image ${\cal C}$ of this immersion, which we lift and appropriately modify on ${\cal X}$.
An important role in this construction is played by Proposition~\ref{Isomorphism}, in which we establish an isomorphism between the residual cohomologies on ${\cal C}$ and the
cohomologies of the structural sheaf of ${\cal C}$. This proposition might be considered
as a step in constructing a Hodge-type decomposition of cohomologies on curves
with singularities, following direction of \cite{Gr},\cite{Ha2}, and \cite{Dl}.
We notice that the proof of Theorem~\ref{ExplicitHodge} generalizes verbatim to the case
of a nonsingular projective complete intersection.\\
\indent
In section \ref{Vanishing} using Theorems~\ref{HomogeneousHodge} and
\ref{ExplicitHodge} we obtain explicit formulas for solutions of $\bar\partial$-equation
and present two explicit versions of the {\it Hodge-Kodaira Vanishing Theorem}
for open Riemann surfaces.

\section{Generalized version of Theorem 1 from \cite{HP4}.}
\label{NewVersion}

\indent
Below we present a generalized version of Theorem 1 from \cite{HP4}, which gives
a Hodge-type decomposition of residual currents of arbitrary homogeneity on reduced
complete intersections in $\C\P^n$. Before formulating this version we introduce definitions
from \cite{HP3} and \cite{HP4}.\\
\indent
Let $V$ be a complete intersection subvariety
\begin{equation}\label{Variety}
V=\left\{z\in \C\P^n:\ P_1(z)=\cdots=P_m(z)=0\right\}
\end{equation}
of dimension $n-m$ in $\C\P^n$ defined by a collection $\left\{P_k\right\}_{k=1}^m$
of homogeneous polynomials. Let
$$\left\{U_{\alpha}=\left\{z\in\C\P^n:\ z_{\alpha}\neq 0\right\}\right\}_{\alpha=0}^n$$
be the standard covering of $\C\P^n$, and let
$${\bf F}^{(\alpha)}(z)=\left[\begin{tabular}{c}
$F^{(\alpha)}_1(z)$\vspace{0.1in}\\
\vdots\vspace{0.1in}\\
$F^{(\alpha)}_m(z)$
\end{tabular}\right]
=\left[\begin{tabular}{c}
$P_1(z)/z_{\alpha}^{\deg P_1}$\vspace{0.1in}\\
\vdots\vspace{0.1in}\\
$P_m(z)/z_{\alpha}^{\deg P_m}$
\end{tabular}\right]$$
be collections of nonhomogeneous polynomials satisfying
$${\bf F}^{(\alpha)}(z)=A_{\alpha\beta}(z)\cdot{\bf F}^{(\beta)}(z)
=\left[\begin{tabular}{ccc}
$\left(z_{\beta}/z_{\alpha}\right)^{\deg P_1}$&$\cdots$&0\vspace{0.1in}\\
\vdots&$\ddots$&\vdots\vspace{0.1in}\\
0&$\cdots$&$\left(z_{\beta}/z_{\alpha}\right)^{\deg P_m}$
\end{tabular}\right]
\cdot{\bf F}^{(\beta)}(z)$$
on $U_{\alpha\beta}=U_{\alpha}\cap U_{\beta}$.\\
\indent
Following \cite{Gr} and \cite{Ha2} we consider a line bundle ${\cal L}$ on $V$
with transition functions
$$l_{\alpha\beta}(z)=\det A_{\alpha\beta}
=\left(\frac{z_{\beta}}{z_{\alpha}}\right)^{\sum_{k=1}^m\deg P_k}$$
on $U_{\alpha\beta}$ and the {\it dualizing bundle} on the complete intersection subvariety $V$
\begin{equation}\label{Dualizing}
\omega^{\circ}_V=\omega_{\C\P^n}\otimes{\cal L},
\end{equation}
where $\omega_{\C\P^n}$ is the canonical bundle on $\C\P^n$.\\
\indent
For $q=1,\dots,n-m$ we denote by ${\cal E}^{(n,n-m-q)}\left(V,{\cal L}(-\ell)\right)
={\cal E}^{(0,n-m-q)}\left(V,\omega^{\circ}_V(-\ell)\right)$ the space
of $C^{\infty}$ differential forms of bidegree $(n,n-m-q)$ with coefficients in
${\cal L}\otimes{\cal O}(-\ell)$, i.e. the space of collections of forms
$$\left\{\gamma_{\alpha}\in
{\cal E}^{(n,n-m-q)}\left(U_{\alpha}\right)\right\}_{\alpha=0}^n$$
satisfying
\begin{equation}\label{Transition}
\gamma_{\alpha}=l_{\alpha\beta}\cdot\left(\frac{z_{\beta}}{z_{\alpha}}\right)^{-\ell}
\gamma_{\beta}+\sum_{k=1}^mF^{(\alpha)}_k\cdot\gamma^{\alpha\beta}_k\
\mbox{on}\ U_{\alpha}\cap U_{\beta}.
\end{equation}
\indent
Then, following \cite{CH, HP3, Pa} we define residual currents and $\bar\partial$-closed
residual currents on $V$. By a residual current $\phi\in C_R^{(0,q)}(V,{\cal O}(\ell))$
of homogeneity $\ell$ we call a collection $\left\{\Phi_{\alpha}^{(0,q)}\right\}_{\alpha=0}^n$
of $C^{\infty}$ differential forms satisfying equalities
\begin{equation}\label{ResidualCurrent}
\Phi_{\alpha}=\left(\frac{z_{\beta}}{z_{\alpha}}\right)^{\ell}\Phi_{\beta}
+\sum_{k=1}^mF^{(\alpha)}_k
\cdot\Omega^{(\alpha\beta)}_k\ \mbox{on}\ U_{\alpha}\cap U_{\beta},
\end{equation}
acting on $\gamma\in {\cal E}^{(n,n-m-q)}\left(V,{\cal L}(-\ell)\right)$ by the formula
\begin{equation}\label{GlobalCurrent}
\langle\phi,\gamma\rangle=\sum_{\alpha}
\int_{U_{\alpha}}\vartheta_{\alpha}\gamma_{\alpha}\wedge\Phi_{\alpha}
\bigwedge_{k=1}^m\bar\partial\frac{1}{F^{(\alpha)}_k}
\stackrel{\text{def}}{=}\lim_{t\to 0}\sum_{\alpha}\int_{T^{\epsilon(t)}_{\alpha}}
\vartheta_{\alpha}\frac{\gamma_{\alpha}\wedge\Phi_{\alpha}}{\prod_{k=1}^m F^{(\alpha)}_k},
\end{equation}
where $\left\{\vartheta_{\alpha}\right\}_{\alpha=0}^n$ is a partition of unity
subordinate to the covering $\left\{U_{\alpha}\right\}_{\alpha=0}^n$, and
the limit in the right-hand side of \eqref{GlobalCurrent} is taken along an admissible path
in the sense of Coleff-Herrera \cite{CH}, i.e. an analytic map
$\epsilon:[0,1]\to \R^m$ satisfying conditions
\begin{equation}\label{admissible}
\begin{cases}
\lim_{t\to 0}\epsilon_m(t)=0,\\
{\dis \lim_{t\to 0}\frac{\epsilon_j(t)}{\epsilon^l_{j+1}(t)}=0,\
\mbox{for any}\ l\in\N }\ \text{and}\ j=1,\dots,m-1,
\end{cases}
\end{equation}
and
\begin{equation}\label{Tube}
T^{\epsilon(t)}_{\alpha}
=\left\{z\in U_{\alpha}:\ \left|F^{(\alpha)}_k(z)\right|=\epsilon_k(t)\ \text{for}\ k=1,\dots,m\right\}.
\end{equation}
\indent
A residual current $\phi$ of homogeneity $\ell$ we call $\bar\partial$-closed
$\left(\text{denoted}\ \phi\in Z_R^{(0,q)}(V,{\cal O}(\ell))\right)$,
if it satisfies the following condition
\begin{equation}\label{Closed}
\bar\partial\Phi_{\alpha}
=\sum_{k=1}^mF^{(\alpha)}_k\cdot\Omega^{(\alpha)}_k\ \mbox{on}\ U_{\alpha}.
\end{equation}

\indent
Below we present an extended version of Theorem 1 from \cite{HP4} that is used in this article.

\begin{theorem}\label{HomogeneousHodge} Let $V\subset \C\P^n$ be a reduced complete
intersection subvariety as in \eqref{Variety}. Then
\begin{itemize}
\item[(i)]
for an arbitrary $\phi\in Z_R^{(0,q)}\left(V, {\cal O}(\ell)\right)$ the following
representation holds
\begin{equation}\label{HomogeneousHomotopy}
\phi=\bar\partial I_q[\phi]+L_q[\phi],
\end{equation}
where $L_q[\phi]=0$ if $1\leq q<n-m$,
$L_{n-m}[\phi]\in  Z_R^{(0,n-m)}\left(V, {\cal O}(\ell)\right)$ is defined by formula
\begin{multline}\label{HomogeneousL}
L_{n-m}[\phi]=\sum_{0\leq r\leq d-n-1-\ell}C(n,m,d,r)\lim_{t\to 0}
\int_{\left\{|\zeta|=1,\left\{|P_k(\zeta)|=\epsilon_k(t)\right\}_{k=1}^m\right\}}
\langle{\bar z}\cdot\zeta\rangle^r\cdot\frac{\phi(\zeta)}{\prod_{k=1}^mP_k(\zeta)}\\
\bigwedge\det\left[{\bar z}\ \overbrace{Q(\zeta,z)}^{m}\
\overbrace{d{\bar z}}^{n-m}\right]\wedge\omega(\zeta)
\end{multline}
with $d=\sum_{k=1}^m\deg P_k$, and the current
$I_q[\phi]\in C^{(0,q-1)}\left(V, {\cal O}(\ell)\right)$ is defined by formula
\begin{multline}\label{HomogeneousI}
I_q[\phi]=C(n,q,m)\lim_{t\to 0}
\int_{\left\{|\zeta|=1,\left\{|P_k(\zeta)|=\epsilon_k(t)\right\}_{k=1}^m\right\}}
\frac{\phi(\zeta)}{\prod_{k=1}^mP_k(\zeta)}\\
\bigwedge\det\left[\frac{\bar z}{B^*(\zeta,z)}\ \frac{\bar\zeta}{B(\zeta,z)}\
\overbrace{Q(\zeta,z)}^{m}\
\overbrace{\frac{d{\bar z}}{B^*(\zeta,z)}}^{q-1}\
\overbrace{\frac{d{\bar\zeta}}{B(\zeta,z)}}^{n-m-q}\right]\wedge\omega(\zeta),
\end{multline}
where functions $\left\{Q_k^i(\zeta,z)\right\}$ for $k=1,\dots,m,\ i=0,\dots,n$ satisfy
\begin{equation}\label{QFunctions}
\left\{\begin{array}{ll}
P_k(\zeta)-P_k(z)=\sum_{i=0}^nQ_k^i(\zeta,z)\cdot\left(\zeta_i-z_i\right),\vspace{0.1in}\\
Q_k^i(\lambda\zeta,\lambda z)=\lambda^{\deg{P_k}-1}\cdot Q_k^i(\zeta,z)\
\mbox{for}\ \lambda\in\C,
\end{array}\right.
\end{equation}
and
$$B^*(\zeta,z)=\sum_{j=0}^n{\bar z}_j\cdot\left(\zeta_j-z_j\right),
\hspace{0.2in}B(\zeta,z)=\sum_{j=0}^n{\bar\zeta}_j\cdot\left(\zeta_j-z_j\right);$$
\item[(ii)]
a $\bar\partial$-closed residual current
$\phi\in Z_R^{(0,n-m)}\left(V, {\cal O}(\ell)\right)$
is $\bar\partial$-exact, i.e. there exists a current
$\psi\in C^{(0,n-m-1)}\left(V, {\cal O}(\ell)\right)$ such that $\phi=\bar\partial\psi$, iff
\begin{equation}\label{HodgeCondition}
L_{n-m}[\phi]=0.
\end{equation}
\end{itemize}
\end{theorem}
\medskip
We present here a sketch of proof of the Theorem. In this sketch we are concerned
with extending the validity of all lemmas and propositions comprising the proof of Theorem 1
in \cite{HP4} to the case of nonzero homogeneity.\\
\indent
In the lemma below we prove that the operators $L$ and $I$ defined in \cite{HP4}
preserve homogeneity, validating equalities (2.6), (2.11) and 
Proposition 2.3 from \cite{HP4} in the case of nonzero homogeneity.\\
\begin{lemma}\label{Homogeneity} If $\phi\in C_R^{(0,q)}(V,{\cal O}(\ell))$
is a residual current defined by a differential form $\Phi$ in a neighborhood of $V$ satisfying
\begin{equation}\label{lHomogeneity}
\Phi(\lambda\zeta)=\lambda^{\ell}\cdot\Phi(\zeta),
\end{equation}
then for the forms $L_{n-m}\left[\Phi\right]$ and $I_q\left[\Phi\right]$
defined by operators $L_{n-m}$ and $I_q$ in formulas (3.24) and (4.22) of \cite{HP4}
respectively, the following equalities hold
\begin{equation}\label{HomogeneityCondition}
\begin{aligned}
&L_{n-m}\left[\Phi\right](\lambda z)=\lambda^{\ell}\cdot L_{n-m}\left[\Phi\right](z),\\
&I_q\left[\Phi\right](\lambda z)=\lambda^{\ell}\cdot I_q\left[\Phi\right](z).
\end{aligned}
\end{equation}
\end{lemma}
\begin{proof}
To prove the preservation of homogeneity for operator $L_{n-m}$ we use
the following equality for $|\lambda|=1$, after changing variables to $\zeta=\lambda w$
\begin{multline*}
\int_{\left\{|\zeta|=1,\left\{|P_k(\zeta)|=\epsilon_k(t)\right\}_{k=1}^m\right\}}
\langle\bar{\lambda}{\bar z}\cdot\zeta\rangle^r
\cdot\frac{\Phi(\zeta)}{\prod_{k=1}^mP_k(\zeta)}
\bigwedge\det\left[\bar{\lambda}{\bar z}\ \overbrace{Q(\zeta,\lambda z)}^{m}\
\overbrace{\bar{\lambda}d{\bar z}}^{n-m}\right]\wedge\omega(\zeta)\\
=\int_{\left\{|\zeta|=1,\left\{|P_k(w)|=\epsilon_k(t)\right\}_{k=1}^m\right\}}
\langle\bar{\lambda}{\bar z}\cdot\lambda w\rangle^r
\cdot\frac{\Phi(\lambda w)}{\prod_{k=1}^mP_k(\lambda w)}
\bigwedge\det\left[\bar{\lambda}{\bar z}\ \overbrace{Q(\lambda w,\lambda z)}^{m}\
\overbrace{\bar{\lambda}d{\bar z}}^{n-m}\right]\wedge\omega(\lambda w)\\
=\lambda^{\ell-d}\bar{\lambda}^{n-m+1}\lambda^{d-m}\lambda^{n+1}
\int_{\left\{|\zeta|=1,\left\{|P_k(w)|=\epsilon_k(t)\right\}_{k=1}^m\right\}}
\langle{\bar z}\cdot w\rangle^r
\cdot\frac{\Phi(w)}{\prod_{k=1}^mP_k(w)}\\
\bigwedge\det\left[\bar{\lambda}{\bar z}\ \overbrace{Q(w,z)}^{m}\
\overbrace{d{\bar z}}^{n-m}\right]\wedge\omega(w)\\
=\lambda^{\ell}
\int_{\left\{|\zeta|=1,\left\{|P_k(w)|=\epsilon_k(t)\right\}_{k=1}^m\right\}}
\langle{\bar z}\cdot w\rangle^r
\cdot\frac{\Phi(w)}{\prod_{k=1}^mP_k(w)}
\bigwedge\det\left[\bar{\lambda}{\bar z}\ \overbrace{Q(w,z)}^{m}\
\overbrace{d{\bar z}}^{n-m}\right]\wedge\omega(w),\\
\end{multline*}
where $d=\sum_{k=1}^m\deg P_k$.\\
\indent
Similarly, we obtain equality
\begin{multline*}
\int_{\left\{|\zeta|=1,\left\{|P_k(w)|=\epsilon_k(t)\right\}_{k=1}^m\right\}}
\frac{\Phi(\lambda w)}{\prod_{k=1}^mP_k(\lambda w)}\\
\bigwedge\det\left[\frac{{\bar\lambda}\bar z}
{B^*(\lambda w,\lambda z)}\ \frac{\bar\lambda{\bar w}}{B(\lambda w,\lambda z)}\
\overbrace{Q(\lambda w,\lambda z)}^{m}\
\overbrace{\frac{{\bar\lambda}d{\bar z}}{B^*(\lambda w,\lambda z)}}^{q-1}\
\overbrace{\frac{{\bar\lambda}d{\bar w}}{B(\lambda w,\lambda z)}}^{n-m-q}\right]
\wedge\omega(\lambda w)\\
=\lambda^{\ell-d}\lambda^{d-m}{\bar\lambda}^{n-m+1}\lambda^{n+1}
\int_{\left\{|\zeta|=1,\left\{|P_k(w)|=\epsilon_k(t)\right\}_{k=1}^m\right\}}
\frac{\Phi(w)}{\prod_{k=1}^mP_k(w)}\\
\bigwedge\det\left[\frac{\bar z}
{B^*(w,z)}\ \frac{{\bar w}}{B(w,z)}\
\overbrace{Q(w,z)}^{m}\
\overbrace{\frac{d{\bar z}}{B^*(w,z)}}^{q-1}\
\overbrace{\frac{d{\bar w}}{B(w,z)}}^{n-m-q}\right]
\wedge\omega(w)\\
=\lambda^{\ell}\int_{\left\{|\zeta|=1,\left\{|P_k(w)|=\epsilon_k(t)\right\}_{k=1}^m\right\}}
\frac{\Phi(w)}{\prod_{k=1}^mP_k(w)}\\
\bigwedge\det\left[\frac{\bar z}
{B^*(w,z)}\ \frac{{\bar w}}{B(w,z)}\
\overbrace{Q(w,z)}^{m}\
\overbrace{\frac{d{\bar z}}{B^*(w,z)}}^{q-1}\
\overbrace{\frac{d{\bar w}}{B(w,z)}}^{n-m-q}\right]
\wedge\omega(w).
\end{multline*}
\end{proof}

\indent  To include the case of nonzero homogeneity Lemmas 3.2 and 3.3 in \cite{HP4}
have to be reformulated. In particular, Lemma 3.2 has to be replaced by the following Lemma.

\begin{lemma}\label{ResidueReduction} Let $V\subset \C\P^n$ be a reduced complete intersection
subvariety as in \eqref{Variety},
let $U\supset V$ be an open neighborhood
of $V$ in $\C\P^n$, and let $\Phi\in {\cal E}_c^{(0,n-m)}(U\cap U_{\alpha})$
be a differential form of homogeneity $\ell$ on $U\cap U_{\alpha}$ for
some $\alpha\in (0,\dots,n)$.\\
\indent
Then formula
\begin{equation}\label{SphericalIntegral}
\lim_{t\to 0}\int_{\left\{|\zeta|=\tau,\left\{|P_k(\zeta)|=\epsilon_k(t)\right\}_{k=1}^m\right\}}
\langle{\bar z}\cdot\zeta\rangle^r\cdot\frac{\Phi(\zeta)}{\prod_{k=1}^mP_k(\zeta)}
\wedge\det\left[{\bar z}\ \overbrace{Q(\zeta,z)}^{m}\
\overbrace{d{\bar z}}^{n-m}\right]\wedge\omega(\zeta),
\end{equation}
where $\left\{\epsilon_k(t)\right\}_{k=1}^m$ is an admissible path,
defines a differential form of homogeneity $\ell$ on $U$, real analytic with respect to $z$.\\
\indent
If $\Phi(\zeta)=\sum_{k=1}^m F^{(\alpha)}_k(\zeta)\Omega_k(\zeta)$ with
$\Omega_k\in {\cal E}_c^{(0,n-m)}(U\cap U_{\alpha})$,
then the limit in \eqref{SphericalIntegral} is equal to zero.
\end{lemma}
\begin{proof}
Preservation of homogeneity follows from the first equality in \eqref{HomogeneityCondition}
in Lemma~\ref{Homogeneity}. As in Lemma 3.2 in \cite{HP4} without loss of generality
we can consider only the case $\alpha=0$, i.e.
$\Phi\in {\cal E}_c^{(0,n-m)}\left(U\cap \{\zeta_0\neq 0\}\right)$.
Then formula (3.15) in \cite{HP4} for a form $\Phi$ satisfying
equality \eqref{lHomogeneity} has to be replaced by the following formula
\begin{multline}\label{LocalIntegralEquality}
\int_{\left\{|\zeta|=\tau,\left\{|P_k(\zeta)|=\epsilon_k(t)\right\}_{k=1}^m\right\}}
\left({\bar z}_0+\sum_{j=1}^n{\bar z}_j\cdot w_j\right)^r\cdot
\frac{\Phi(\zeta)}{\prod_{k=1}^mP_k(\zeta)}\\
\wedge\det\left[{\bar z}\ \overbrace{Q(\zeta,z)}^{m}\ \overbrace{d{\bar z}}^{n-m}\right]
\wedge \left(\zeta_0^{n+r}d\zeta_0\right)\bigwedge_{j=1}^n dw_j\\
=\int_{\left\{|\zeta|=\tau,\left\{|P_k(\zeta)|=\epsilon_k(t)\right\}_{k=1}^m\right\}}
\left({\bar z}_0+\sum_{j=1}^n{\bar z}_j\cdot w_j\right)^r\cdot
\frac{\Phi(\zeta)}{\prod_{k=1}^mP_k(\zeta)}\\
\wedge\det\left[{\bar z}\ \overbrace{Q(\zeta,z)}^{m}\ \overbrace{d{\bar z}}^{n-m}\right]
\wedge \left(\zeta_0^{n+r}d\zeta_0\right)\bigwedge_{j=1}^n dw_j\\
=\int_{\left\{|\zeta|=\tau,\left\{|F_k(w)|\cdot\chi_k(w)=\epsilon_k(t)\right\}_{k=1}^m\right\}}
\left({\bar z}_0+\sum_{j=1}^n{\bar z}_j\cdot w_j\right)^r
\left(\zeta_0^{n+r+\ell-\sum_{k=1}^m \deg P_k}d\zeta_0\right)\\
\times\frac{\Phi(w)}{\prod_{k=1}^m F_k(w)}
\wedge\det\left[{\bar z}\ \overbrace{Q(e^{i\phi_0},w,z)}^{m}\
\overbrace{d{\bar z}}^{n-m}\right]\bigwedge_{j=1}^n dw_j\\
=i\int_0^{2\pi}e^{i\left(n+r+\ell-\sum_{k=1}^m \deg P_k+1\right)\phi_0}d\phi_0
\int_{\left\{w\in U_0,\left\{|F_k(w)|\cdot\chi_k(w)=\epsilon_k(t)\right\}_{k=1}^m\right\}}
\left({\bar z}_0+\sum_{j=1}^n{\bar z}_j\cdot w_j\right)^r\\
\times\rho_0(w)^{n+r+\ell-\sum_{k=1}^m \deg P_k+1}
\cdot\frac{\Phi(w)}{\prod_{k=1}^m F_k(w)}
\wedge\det\left[{\bar z}\ \overbrace{Q(e^{i\phi_0},w,z)}^{m}\
\overbrace{d{\bar z}}^{n-m}\right]\bigwedge_{j=1}^n dw_j,
\end{multline}
where we used nonhomogeneous coordinates
$\left\{w_i=\zeta_i/\zeta_0\right\}_{i=1}^n$
in the subset $\*S^{2n+1}(\tau)\setminus\left\{\zeta_0=0\right\}$
of the sphere $\*S^{2n+1}(\tau)$ of radius $\tau$ in $\C^{n+1}$, notation
$\zeta_0=\rho_0(w)\cdot e^{i\phi_0}$ with
$$\rho_0(w)=\frac{\tau}{\sqrt{1+\sum_{i=1}^n|w_i|^2}}$$
on $\*S^{2n+1}(\tau)$, nonhomogeneous polynomials
\begin{equation*}
F_k(w)=P_k(\zeta)/\zeta_0^{\deg P_k}
\end{equation*}
in $\*S^{2n+1}(\tau)\setminus\left\{\zeta_0=0\right\}$,
and equality \eqref{lHomogeneity}.\\
\indent
The last statement of the Lemma follows as in Lemma 3.2 of \cite{HP4} from application of Theorem 1.7.6(2)
in \cite{CH} to the interior integral in the right-hand side of \eqref{LocalIntegralEquality}.
\end{proof}

\indent
Lemma 3.3 has to be replaced by the following Lemma.

\begin{lemma}\label{ZeroIntegrals} Let $\phi\in Z_R^{(0,q)}(V,{\cal O}(\ell))$
be a $\bar\partial$-closed residual current defined
by a collection of forms $\left\{\Phi^{(0,n-m)}_{\alpha}\right\}_{\alpha=0}^n$
of homogeneity $\ell$ on a neighborhood $U$ of the reduced subvariety $V$
as in \eqref{Variety} satisfying \eqref{ResidualCurrent} and \eqref{Closed},
and let $\Phi(\zeta)=\sum_{\alpha=0}^n\vartheta_{\alpha}(\zeta)\Phi_{\alpha}(\zeta)$
be a differential form of homogeneity $\ell$ on $U$.\\
\indent
Then for an arbitrary
$\gamma\in {\cal E}^{(n,0)}\left(V,{\cal L}(-\ell)\right)$
the equality
\begin{multline}\label{ZeroIntegral}
\lim_{\tau\to 0}\int_{T^{\delta(\tau)}_{\beta}}
\frac{\gamma(z)}{\prod_{k=1}^m F^{(\beta)}_k(z)}
\bigwedge\left(\lim_{t\to 0}\int_{\left\{|\zeta|=1,\left\{|P_k(\zeta)|=\epsilon_k(t)\right\}_{k=1}^m\right\}}
\langle{\bar z}\cdot\zeta\rangle^r\right.\\
\times\frac{\Phi(\zeta)}{\prod_{k=1}^mP_k(\zeta)}
\left.\wedge\det\left[{\bar z}\ \overbrace{Q(\zeta,z)}^{m}\
\overbrace{d{\bar z}}^{n-m}\right]\wedge\omega(\zeta)\right)=0
\end{multline}
holds unless
\begin{equation}\label{IndexCondition}
r\leq\sum_{k=1}^m\deg P_k-\ell-n-1.
\end{equation}
\end{lemma}
\begin{proof} We notice that for all values of $a>0$
and $\epsilon=\left(\epsilon_1,\dots,\epsilon_k\right)$ the sets
$$S( a)=\left\{\zeta\in \*S^{2n+1}(a):
\left\{ \left|P_k(\zeta)\right|=\epsilon_k\cdot a^{\deg P_k}\right\}_{k=1}^m\right\}$$
are real analytic subvarieties of $\*S^{2n+1}(a)$ of real dimension $2n+1-m$
satisfying
\begin{equation}\label{VolumeEstimate}
c\cdot a^{2n+1-m}\cdot\text{Volume}\left(S(1)\right)
<\text{Volume}_{2n+1-m}\left(S(a)\right)
<C\cdot a^{2n+1-m}\cdot\text{Volume}\left(S(1)\right).
\end{equation}
\indent
We denote
$$\Phi_{\alpha}(\zeta,z)=\langle{\bar z}\cdot\zeta\rangle^r\cdot
\frac{\Phi_{\alpha}(\zeta)}{\prod_{k=1}^mP_k(\zeta)}
\det\left[{\bar z}\ \overbrace{Q(\zeta,z)}^{m}\ \overbrace{d{\bar z}}^{n-m}\right]
\wedge\omega(\zeta),$$
and apply the Stokes' formula to the differential form
\begin{equation}\label{betaform}
\beta(\zeta,z)=\sum_{\alpha=0}^n\vartheta_{\alpha}(\zeta)\Phi_{\alpha}(\zeta,z)
=\sum_{\alpha=0}^n\langle{\bar z}\cdot\zeta\rangle^r
\cdot\frac{\vartheta_{\alpha}(\zeta)\Phi_{\alpha}(\zeta)}
{\prod_{k=1}^mP_k(\zeta)}
\det\left[{\bar z}\ \overbrace{Q(\zeta,z)}^{m}\ \overbrace{d{\bar z}}^{n-m}\right]\wedge\omega(\zeta)
\end{equation}
on the variety
$$\left\{\zeta\in \C^{n+1}:
\left\{ \left|P_k(\zeta)\right|=\epsilon_k\cdot|\zeta|^{\deg P_k}\right\}_{k=1}^m,\
a<|\zeta|<1\right\}$$
with the boundary
\begin{equation*}
\Big\{\zeta:|\zeta|=a,\
\left\{|P_k(\zeta)|=\epsilon_k\cdot a^{\deg P_k}\right\}_{k=1}^m\Big\}
\bigcup\Big\{\zeta:|\zeta|=1,\ \left\{|P_k(\zeta)|=\epsilon_k\right\}_{k=1}^m\Big\}.
\end{equation*}
Then using equality \eqref{Closed}
we obtain the equality
\begin{multline}\label{epsilonStokes}
\int_{\left\{|\zeta|=1, \left\{|P_k(\zeta)|=\epsilon_k(t)\right\}_{k=1}^m\right\}}\beta(\zeta,z)
-\int_{\left\{|\zeta|=a,\
\left\{|P_k(\zeta)|=\epsilon_k(t)\cdot a^{\deg P_k}\right\}_{k=1}^m\right\}}\beta(\zeta,z)\\
=\sum_{\alpha=0}^n\sum_{j=1}^m\int_{a}^1d\tau
\int_{\left\{|\zeta|=\tau, \left\{|P_k(\zeta)|=\epsilon_k(t)\cdot\tau^{\deg P_k}\right\}_{k=1}^m\right\}}
F^{(\alpha)}_j(\zeta)\cdot\beta_j^{(\alpha)}(\zeta,z)\\
+\sum_{\alpha=0}^n\int_{a}^1d\tau
\int_{\left\{|\zeta|=\tau, \left\{|P_k(\zeta)|=\epsilon_k(t)\cdot\tau^{\deg P_k}\right\}_{k=1}^m\right\}}d|\zeta|\interior
\Big(\bar\partial\vartheta_{\alpha}(\zeta)\wedge\Phi_{\alpha}(\zeta,z)\Big)
\end{multline}
for arbitrary $t$ and $0<a<1$.\\
\indent
Using estimate \eqref{VolumeEstimate} and the homogeneity property
\begin{equation}\label{PhiHomogeneity}
\Phi_{\beta}(t\cdot\zeta)={\bar t}^{-(n-m)}\cdot t^{\ell}\cdot\Phi_{\beta}(\zeta)
\end{equation}
of the coefficients of $\Phi^{(0,n-m)}$ from Proposition 1.1 in \cite{HP1} we obtain that if
\begin{equation}\label{IndexPositive}
r+n+1+\ell-\sum_{k=1}^m\deg P_k >0,
\end{equation}
then
\begin{equation*}
\left|\int_{\left\{|\zeta|= a,\
\left\{|P_k(\zeta)|=\epsilon_k(t)\cdot a^{\deg P_k}\right\}_{k=1}^m\right\}}
\beta(\zeta,z)\right|
<C^{\prime}(\epsilon)\cdot a^{r+2n+1-m-(n-m)
+\ell-\sum_{k=1}^m\deg P_k}\longrightarrow 0
\end{equation*}
as $t$ is fixed and $a\to 0$.\\
\indent
For the first sum of integrals in the right-hand side of \eqref{epsilonStokes} we have
\begin{multline}\label{AreaIntegralOne}
\Bigg|\int_{ a}^1d\tau\int_{\left\{|\zeta|=\tau,\
\left\{|P_k(\zeta)|=\epsilon_k(t)\cdot\tau^{\deg P_k}\right\}_{k=1}^m\right\}}
F^{(\alpha)}_j(\zeta)\cdot\beta_j^{(\alpha)}(\zeta,z)\Bigg|\\
<C\int_{ a}^1d\tau\cdot\tau^{r+2n+1-m-(n-m)+\ell-\sum_{k=1}^m\deg P_k}\\
\times\int_{\left\{|\zeta|=\tau,\
\left\{|P_k(\zeta)|=\epsilon_k(t)\cdot\tau^{\deg P_k}\right\}_{k=1}^m\right\}}
F^{(\alpha)}_j(\zeta)\cdot\left(d|\zeta|\interior\beta_j^{(\alpha)}(\zeta,z)\right)\\
<C^{\prime}(\epsilon)\frac{\left(1- a^{r+n+2+\ell-\sum_{k=1}^m\deg P_k}\right)}
{r+n+2+\ell-\sum_{k=1}^m\deg P_k}\longrightarrow 0
\end{multline}
as $t\to 0$, since $a<1$, condition \eqref{IndexPositive} is satisfied,
and $C^{\prime}(\epsilon)\to 0$ as $t\to 0$ by Lemma~\ref{ResidueReduction}.\\
\indent
For the second sum of integrals in the right-hand side of \eqref{epsilonStokes}
we use equality
\begin{multline*}
\lim_{t\to 0}
\sum_{\left\{\beta:\ U_{\beta}\cap U_{\alpha}\neq\emptyset\right\}}\int_{a}^1d\tau
\int_{\left\{\zeta_{\alpha}\neq 0,|\zeta|=\tau, \left\{|P_k(\zeta)|
=\epsilon_k(t)\cdot\tau^{\deg P_k}\right\}_{k=1}^m\right\}}
d|\zeta|\interior \Big(\bar\partial\vartheta_{\beta}(\zeta)\wedge\Phi_{\beta}(\zeta,z)\Big)\\
=\lim_{t\to 0}
\int_{a}^1d\tau\sum_{\left\{\beta:\ U_{\beta}\cap U_{\alpha}\neq\emptyset\right\}}
\int_{\left\{\zeta_{\alpha}\neq 0,|\zeta|=\tau, \left\{|P_k(\zeta)|
=\epsilon_k(t)\cdot\tau^{\deg P_k}\right\}_{k=1}^m\right\}}
\langle{\bar z}\cdot\zeta\rangle^r\\
\times d|\zeta|\interior\left(\bar\partial\vartheta_{\beta}(\zeta)
\wedge\frac{\Phi_{\beta}(\zeta)}{\left[\prod_{k=1}^mP_k(\zeta)\right]\big|_{\beta}}
\bigwedge\det\left[{\bar z}\ \overbrace{Q(\zeta,z)}^{m}\
\overbrace{d{\bar z}}^{n-m}\right]\wedge\omega(\zeta)\right)\\
=\lim_{t\to 0}
\int_{a}^1d\tau\sum_{\left\{\beta:\ U_{\beta}\cap U_{\alpha}\neq\emptyset\right\}}
\int_{\left\{\zeta_{\alpha}\neq 0,|\zeta|=\tau, \left\{|P_k(\zeta)|
=\epsilon_k(t)\cdot\tau^{\deg P_k}\right\}_{k=1}^m\right\}}
\langle{\bar z}\cdot\zeta\rangle^r\\
\times d|\zeta|\interior\left(\bar\partial\vartheta_{\beta}(\zeta)\big|_{U_{\alpha}}
\wedge\frac{\Phi_{\alpha}(\zeta)}
{\left[\prod_{k=1}^mP_k(\zeta)\right]\big|_{\alpha}}
\wedge\det\left[{\bar z}\ \overbrace{Q(\zeta,z)}^{m}\
\overbrace{d{\bar z}}^{n-m}\right]\wedge\omega(\zeta)\right)\\
=\lim_{t\to 0}
\int_{a}^1d\tau\int_{\left\{\zeta_{\alpha}\neq 0,|\zeta|=\tau, \left\{|P_k(\zeta)|
=\epsilon_k(t)\cdot\tau^{\deg P_k}\right\}_{k=1}^m\right\}}
\langle{\bar z}\cdot\zeta\rangle^r\\
\times d|\zeta|\interior
\left(\left[\sum_{\left\{\beta:\ U_{\beta}\cap U_{\alpha}\neq\emptyset\right\}}
\bar\partial\vartheta_{\beta}(\zeta)\big|_{U_{\alpha}}\right]
\wedge\frac{\Phi_{\alpha}(\zeta)}
{\left[\prod_{k=1}^mP_k(\zeta)\right]\big|_{\alpha}}
\bigwedge\det\left[{\bar z}\ \overbrace{Q(\zeta,z)}^{m}\
\overbrace{d{\bar z}}^{n-m}\right]\wedge\omega(\zeta)\right)=0.
\end{multline*}
which follows from equality \eqref{ResidualCurrent}, Lemma~\ref{ResidueReduction},
and the transformation formula for the Grothendieck's residue from Proposition 4.2 in \cite{HP1}
(for isolated singularities see \cite{GH,T}).
\end{proof}
\indent
This completes the sketch of proof of Theorem~\ref{HomogeneousHodge}.

\indent
Below we formulate a corollary of Theorem~\ref{HomogeneousHodge}, which will
be used in what follows. This corollary gives an explicit form of the
{\it Hodge-Kodaira Vanishing Theorem}
for projective complete intersections (see III.7.15 in \cite{Ha2} and \S 2.1 of \cite{GH}.
\begin{corollary}\label{PositiveHomogeneity}
If homogeneity $\ell$ satisfies inequality
\begin{equation}\label{PositiveCondition}
\ell>d-n-1,
\end{equation}
then the operator in the right-hand side of equality \ref{HomogeneousL} is zero,
\begin{equation}\label{ZeroCohomology}
\dim H^{n-m}_R\left(V, {\cal O}(\ell)\right)=0,
\end{equation}
and for any $\phi\in Z_R^{(0,n-m)}\left(V, {\cal O}(\ell)\right)$
\begin{equation}\label{DbarSolution}
\phi=\bar\partial I_{n-m}[\phi].
\end{equation}
\end{corollary}

\section{Explicit Hodge decomposition on nonsingular algebraic curves.}

\indent
In order to construct an explicit Hodge decomposition on an arbitrary nonsingular
algebraic curve ${\cal X}$ we use the existence of an immersion
(see \cite{Kd2}, 1.4 in \cite{GH}, IV.3 in \cite{Ha2})
\begin{equation}\label{varrhoMap}
{\cal X}\stackrel{\varrho}{\rightarrow}\C\P^2,
\end{equation}
such that ${\cal C}=\varrho\left({\cal X}\right)$ is a plane curve with at most nodes as singularities.
Our construction of the sought decomposition on ${\cal X}$ will be based on a similar
decomposition on ${\cal C}$, which exists according to Theorem~\ref{HomogeneousHodge}.\\
\indent
We construct a linear map
$\varrho_*:{\cal E}^{(0,1)}({\cal X})\rightarrow Z_R^{(0,1)}\left({\cal C}\right)$.
Here and below we use notations $\varrho_*$ and $\varrho^*$ for induced direct
and respectively inverse maps on functions and differential forms.
Since $\varrho$ is biholomorphic everywhere outside of nodal points,
we have to describe $\varrho_*$ only in the neighborhoods
of those points. Let $p\in {\cal C}$ be such nodal point with $p_1,p_2\in {\cal X}$ such that
$\varrho(p_1)=\varrho(p_2)=p$.
Let $z_1,z_2$ be local coordinates at $p\in U\subset \C\P^2$, such that
$z_1(p)=z_2(p)=0$, and such that
$${\cal C}\cap U={\cal C}_1\cup{\cal C}_2,$$
where ${\cal C}_1=\left\{q\in U: z_1(q)=0\right\}$, and
${\cal C}_2=\left\{q\in U: z_2(q)=0\right\}$.\\
\indent
Let $\phi\in {\cal E}^{(0,1)}({\cal X})$ be a smooth $\bar\partial$-closed
form on ${\cal X}$ with local representations
\begin{equation*}
\phi\Big|_{p_1}=\phi_1(z)d{\bar z}_1,\hspace{0.2in}\phi\Big|_{p_2}=\phi_2(z)d{\bar z}_2
\end{equation*}
on ${\cal X}$. Let $\widetilde{\phi}_1,\widetilde{\phi}_2$ be extensions of $\phi_1$
and $\phi_2$ to $U$ such that
\begin{equation*}
\widetilde{\phi_j}\Big|_{{\cal C}_j}=\phi_j,\hspace{0.1in}
\bar\partial_{z_k}\widetilde{\phi_j}=0\ \text{for}\ k\neq j.
\end{equation*}
Then the differential form
\begin{equation}\label{DirectImage}
\phi_*=\widetilde{\phi}_1d{\bar z}_1+\widetilde{\phi}_2d{\bar z}_2
\end{equation}
defines a $\bar\partial$-closed residual current on ${\cal C}$
in the neighborhood $U\ni p$. Current $\phi_*$ is $\bar\partial$-exact in $U$ since
for the functions $\widetilde{\psi}_1$ and $\widetilde{\psi}_2$ chosen so that
\begin{equation*}
\bar\partial_{z_j}\widetilde{\psi}_j=\widetilde{\phi}_j,\hspace{0.1in}
\bar\partial_{z_k}\widetilde{\psi}_j=0\ \text{for}\ k\neq j,
\end{equation*}
we will have
$$\bar\partial\left(\widetilde{\psi}_1+\widetilde{\psi}_2\right)
=\widetilde{\phi}_1d{\bar z}_1+\widetilde{\phi}_2d{\bar z}_2,$$
and
\begin{equation*}
(\phi_*)^*=\phi.
\end{equation*}
\indent
In the proposition below we identify the spaces of residual cohomologies
$H^{(0,1)}_R({\cal C})$ of curve ${\cal C}$
and the cohomologies $H^1({\cal C},{\cal O})$ of the structural sheaf
${\cal O}$ on ${\cal C}$.

\begin{proposition}\label{Isomorphism}(compare with \cite{Gr, Ha2})
(Isomorphism of cohomologies) Let
$${\cal C}=\left\{z\in\C\P^2: P(z)=0\right\}$$
be a curve in $\C\P^2$ with nodal points as the only singularities.
Then there exists an isomorphism
\begin{equation}\label{imathIsomorphism}
\imath:H^{(0,1)}_R({\cal C})\to H^1({\cal C},{\cal O}).
\end{equation}
\end{proposition}
\begin{proof}
To define $\imath$ we consider an
arbitrary $\phi\in Z^1_R({\cal C},{\cal O})$ represented locally on a cover
$\left\{U_i\right\}_{i=1}^N$ by residual currents
$$\left\{\frac{\Phi_i^{(0,1)}}{F^{(i)}}\right\}_{i=1}^N$$
satisfying $\bar\partial\Phi_i=F^{(i)}\Omega_i^{(0,2)}$. Then, using existence
of $\Psi^{(0,1)}_i\in {\cal E}^{(0,1)}(U_i)$ such that
$$\bar\partial\Psi^{(0,1)}_i=\Omega_i^{(0,2)}$$
we obtain
$$\bar\partial\left(\Phi_i-F^{(i)}\Psi_i\right)=0\hspace{0.1in}\text{in}\ U_i,$$
and therefore
$$\Phi_i=\bar\partial\Theta_i+F^{(i)}\Psi_i$$
for some $\Theta_i\in {\cal E}^{(0,0)}(U_i)$.\\
\indent
From the last equality we obtain that
$$\bar\partial\left(\Theta_i-\Theta_j\right)=F_i\Omega_{ij}\ \mbox{on}\ U_{ij}=U_i\cap U_j,$$
and therefore, by defining 
\begin{equation}\label{Mapimath}
\imath\left\{\Phi_i\right\}=\left\{\Theta_i-\Theta_j\right\}
\end{equation}
we obtain a cocycle in $Z^1({\cal C},{\cal O})$.

\indent
To construct the inverse to $\imath$ we take $\left\{\Theta_{ij}\right\}$ - a cocycle in
$Z^1({\cal C},{\cal O})$. We consider a partition of
unity $\left\{\vartheta_i\right\}_{i=1}^N$ subordinate to some open cover
$\left\{U_i\right\}_{i=1}^N$ of an open neighborhood $U\supset V$.
Without loss of generality we may assume that $\left\{\Theta_{ij}\right\}$ are restrictions of
holomorphic functions on $U_{ij}$. We define the cochain
$$\Theta_i=\sum_{k\neq i,\ U_k\cap U_i\neq\emptyset}\vartheta_k\Theta_{ik},$$
and notice that on $U_{ij}$ we have the equality
\begin{multline*}
\Theta_i-\Theta_j=\sum_{k\neq i,\ U_k\cap U_i\neq\emptyset}\vartheta_k\Theta_{ik}
-\sum_{k\neq j,\ U_k\cap U_j\neq\emptyset}\vartheta_k\Theta_{jk}\\
=\sum_{k\neq i,j\ U_k\cap U_i\neq\emptyset}\vartheta_k\left(\Theta_{ik}-\Theta_{jk}\right)
+\vartheta_j\Theta_{ij}-\vartheta_i\Theta_{ji}\\
=\sum_{k\neq i,j\ U_k\cap U_i\neq\emptyset}\vartheta_k\Theta_{ij}+\vartheta_j\Theta_{ij}
+\vartheta_i\Theta_{ij}+F_i\Omega_{ij}=\Theta_{ij}+F_i\Omega_{ij}
\end{multline*}
with some functions $\Omega_{ij}\in {\cal E}^{(0,0)}(U_{ij})$, and its corollary
\begin{equation}\label{DifferenceIdeal}
\bar\partial\left(\Theta_i-\Theta_j\right)=F_i\bar\partial\Omega_{ij}.
\end{equation}
Therefore, the collection
\begin{equation}\label{ResidualImage}
\left\{\Phi_i=\bar\partial\Theta_i
=\sum_{k\neq i,\ U_k\cap U_i\neq\emptyset}
\Theta_{ik}\bar\partial \vartheta_k\right\}_{i=1}^N\in Z^{(0,1)}_R\left({\cal C}\right)
\end{equation}
defines a $\bar\partial$-closed residual current in a neighborhood of $V$.\\
\indent
If we apply the map $\imath$ from \eqref{Mapimath} to the current
in \eqref{ResidualImage}, then using equality \eqref{DifferenceIdeal} we obtain
$$\imath\left\{\Phi_i\right\}=\left\{\Theta_i-\Theta_j\right\}
=\left\{\Theta_{ij}\right\}\in Z^1({\cal C},{\cal O}).$$
\end{proof}

\indent
Let now $\left\{p^{(i)}\right\}_{i=1}^r$ be the nodal points in ${\cal C}$, and
let points $p^{(i)}_1,p^{(i)}_2\in {\cal X}$ be such that
$\varrho(p^{(i)}_1)=\varrho(p^{(i)}_2)=p^{(i)}\in{\cal C}$.
We consider a collection of paths $\left\{\gamma_i\right\}_{i=1}^r$ in ${\cal X}$
such that $\gamma_i(0)=p^{(i)}_1$ and $\gamma_i(1)=p^{(i)}_2$
and functions $\left\{f_i\right\}_{i=1}^r\in {\cal E}\left({\cal X}\right)$ with
supports in some neighborhoods $U_i\supset \gamma_i$ such that
\begin{equation}\label{fFunctions}
\begin{cases}
\ f_i(p^{(i)}_1)=0,\ f_i(p^{(i)}_2)=1,\vspace{0.1in}\\
\ \partial f_i=0\ \mbox{in}\ V_i\Subset U_i.
\end{cases}
\end{equation}
Then for a path $\gamma_i\subset{\cal X}$ we have
$$\int_{\gamma_i}\bar\partial f_i=f_i(p^{(i)}_2)-f_i(p^{(i)}_1)=1.$$
\indent
For an arbitrary form $\phi\in {\cal E}^{(0,1)}({\cal X})$
and the corresponding residual current $\phi_*$ on ${\cal C}$ we consider
the decomposition on ${\cal C}$ that follows from Theorem~\ref{HomogeneousHodge}:
\begin{equation*}
\phi_*=\bar\partial I[\phi_*]+L[\phi_*],
\end{equation*}
and the lift of this decomposition on ${\cal X}$:
\begin{equation}\label{phiStar}
\phi=(\phi_*)^*=\bar\partial (I[\phi_*])^*+(L[\phi_*])^*.
\end{equation}
From Proposition 4.4 in \cite{HP4} it follows that if $\phi\in {\cal E}^{(0,1)}({\cal X})$, then
$I[\phi_*]\in {\cal E}({\cal C})$, and therefore
\begin{equation*}
(I[\phi_*])^*\in {\cal E}({\cal X}).
\end{equation*}
\indent
We consider the scalar product on ${\cal E}^{(0,1)}\left({\cal X}\right)$
\begin{equation}\label{ScalarProduct}
\left\langle\phi,\psi\right\rangle=\int_{\cal X}\phi\wedge \overline{\psi}
\end{equation}
and assume without loss of generality that the collection of forms
$\left\{(L[(\bar\partial f_i)_*])^*\right\}_{i=1}^r$
is orthonormalized with respect to scalar product in \eqref{ScalarProduct}.
Then for $\phi\in {\cal E}^{(0,1)}\left({\cal X}\right)$
we define for $i=1,\dots,r$,
\begin{equation}\label{aCoefficients}
a_i[\phi]=\int_{\cal X}\phi\wedge \overline{(L[(\bar\partial f_i)_*])^*},
\end{equation}
and consider operators ${\cal L}:{\cal E}^{(0,1)}({\cal X})\to {\cal E}^{(0,1)}({\cal X})$
and ${\cal I}:{\cal E}^{(0,1)}({\cal X})\to {\cal E}({\cal X})$, defined as
\begin{equation}\label{CalLOperator}
{\cal L}[\phi]=(L[\phi_*])^*-\sum_{i=1}^ra_i[\phi](L[(\bar\partial f_i)_*])^*,
\end{equation}
and
\begin{equation}\label{CalIOperator}
{\cal I}[\phi]=\left(I\left[\phi_*-\sum_{i=1}^ra_i[\phi](\bar\partial f_i)_*\right]\right)^*
+\sum_{i=1}^ra_i[\phi]f_i.
\end{equation}

\begin{proposition}\label{CurveHodge}
(Hodge-type decomposition on a nonsingular curve)
Let ${\cal X}$ be a nonsingular curve, and
let ${\cal X}\stackrel{\varrho}{\rightarrow}\C\P^2$ be an immersion of ${\cal X}$ into
$\C\P^2$ with $r$ nodal points. Let $\left\{f_i\right\}_{i=1}^r$ satisfy \eqref{fFunctions}
and let operators ${\cal L}$ and ${\cal I}$ be defined respectively in \eqref{CalLOperator}
and \eqref{CalIOperator}.\\
\indent
Then for the space of $(0,1)$-forms ${\cal E}^{(0,1)}({\cal X})$
\begin{itemize}
\item[(i)]
the following decomposition holds
\begin{equation}\label{HodgeonCurve}
\phi=\bar\partial{\cal I}[\phi]+{\cal L}[\phi],
\end{equation}
\item[(ii)]
a $(0,1)$-form $\phi\in {\cal E}^{(0,1)}({\cal X})$
is $\bar\partial$-exact, iff ${\cal L}[\phi]=0$.
\end{itemize}
\end{proposition}
\begin{proof}
To prove equality \eqref{HodgeonCurve} we use equality \eqref{phiStar} to
obtain for an arbitrary $\bar\partial$-closed form $\phi\in {\cal E}^{(0,1)}({\cal X})$
the equality
\begin{equation*}
\phi-\sum_{i=1}^ra_i[\phi]\bar\partial f_i
=\bar\partial\left(I\left[\phi_*-\sum_{i=1}^ra_i[\phi](\bar\partial f_i)_*\right]\right)^*
+\left(L\left[\phi_*-\sum_{i=1}^ra_i[\phi](\bar\partial f_i)_*\right]\right)^*,
\end{equation*}
which we can rewrite as
\begin{equation*}
\phi=\bar\partial\left[\left(I\left[\phi_*-\sum_{i=1}^ra_i[\phi](\bar\partial f_i)_*\right]\right)^*
+\sum_{i=1}^ra_i[\phi]f_i\right]+{\cal L}[\phi]
=\bar\partial{\cal I}[\phi]+{\cal L}[\phi].
\end{equation*}
\indent
According to Theorem~\ref{HomogeneousHodge} and Proposition~\ref{Isomorphism}
the image of $L^*$ is a subspace of
the space of $\bar\partial$-closed forms - ${\cal E}^{(0,1)}({\cal X})$ of dimension
$\dim H_R^1({\cal C},{\cal O})=p_a({\cal C})$ - the arithmetic genus of ${\cal C}$.
We notice that for any collection $\{c_i\}_{i=1}^r$ there is no $g\in {\cal E}({\cal C})$
such that
\begin{equation*}
\bar\partial g^*=\sum_{i=1}^rc_i\bar\partial f_i,
\end{equation*}
because otherwise we would have
\begin{equation*}
g^*-\sum_{i=1}^rc_if_i=\mbox{const},
\end{equation*}
which contradicts $g^*$ taking the same values at $p^{(i)}_1$ and
$p^{(i)}_2$ for all $i=1,\dots,r$. Therefore,
\begin{equation*}
\dim\mbox{Ker}{\cal L}=\dim\mbox{Span}\left\{(L[(\bar\partial f_i)_*])^*\right\}=r,
\end{equation*}
and
\begin{equation*}
\dim\mbox{Im}\left\{{\cal L}\right\}=\dim\mbox{Im}\left\{L^*\right\}-r=p_a({\cal C})-r.
\end{equation*}
\indent
To prove item (ii) of Proposition~\ref{CurveHodge} we consider a $\bar\partial$-exact form
$\phi=\bar\partial g$ for $g\in {\cal E}({\cal X})$ and assume that
\begin{equation*}
\left(L\left[(\bar\partial g)_*-\sum_{i=1}^ra_i[\bar\partial g](\bar\partial f_i)_*\right]\right)^*
\neq 0.
\end{equation*}
Then from the inequality above it follows that
\begin{equation*}
\dim\left\{{\cal L}\left\{\bar\partial g, g\in{\cal E}({\cal X})\right\}\right\}>r,
\end{equation*}
which, according to Proposition~\ref{Isomorphism}, contradicts the statement
from IV.1 in \cite{Ha2}, that
\begin{equation*}
\dim H^{(0,1)}\left({\cal X}\right)=p_a({\cal C})-r.
\end{equation*}
\end{proof}

\indent
Using Theorem~\ref{HomogeneousHodge} and Proposition~\ref{CurveHodge} we prove
the following version of the Hodge Theorem for smooth algebraic curves,
which gives an explicit formula for solution $R_1[\phi]$ of equation
\begin{equation*}
\bar\partial R_1[\phi]=\phi-H_1[\phi].
\end{equation*}

\begin{theorem}\label{ExplicitHodge}
(Explicit formulas in the Hodge decomposition)
Let ${\cal X}$ be a smooth algebraic curve, let ${\cal L}$ and ${\cal I}$
be operators from \eqref{CalLOperator} and \eqref{CalIOperator} respectively, and
let $\left\{\omega_j\right\}_{j=1}^g$ be an orthonormal basis of
holomorphic $(1,0)$-forms on ${\cal X}$, i.e.
$$\int_{\cal X}\omega_j\wedge\bar\omega_k=\delta_{jk},\ j,k=1,\dots,g.$$
\indent
Then Hodge operators $H_1$ and $R_1$ in decomposition \eqref{ClassicalDecomposition}
admit the following representations
\begin{equation}\label{ExplicitHodgeRepresentation}
\begin{aligned}
&H_1[\phi]=\sum_{j=1}^g\left(\int_{\cal C}\phi\wedge\omega_j\right)\bar\omega_j,\\
&R_1[\phi]={\cal I}\big[\phi+\left({\cal L}-H_1\right)[\phi]\big]+\text{const}.
\end{aligned}
\end{equation}
\end{theorem}
\begin{proof}
Combining decompositions \eqref{ClassicalDecomposition} and \eqref{HomogeneousHomotopy} we
obtain equality
\begin{equation}\label{TwoSides}
\bar\partial\big(R_1[\phi]-{\cal I}[\phi]\big)=\left({\cal L}-H_1\right)[\phi].
\end{equation}
Then, applying Theorem\ref{HomogeneousHodge} to the $\bar\partial$-exact form
$\psi=\left({\cal L}-H_1\right)[\phi]$ we obtain equalities
\begin{equation*}
\psi=\bar\partial {\cal I}[\psi]
\end{equation*}
and
\begin{equation}\label{psiExact}
\left({\cal L}-H_1\right)[\phi]=\bar\partial\big({\cal I}\left({\cal L}-H_1\right)[\phi]\big).
\end{equation}
Finally, from equalities \eqref{TwoSides} and \eqref{psiExact} we obtain
equality
\begin{equation*}
\bar\partial\big(R_1[\phi]-{\cal I}[\phi]-{\cal I}\left({\cal L}-H_1\right)[\phi]\big)=0,
\end{equation*}
which implies the second equality in \eqref{ExplicitHodgeRepresentation}.
\end{proof}

\section{Explicit solution of $\bar\partial$-equation on affine curves.}\label{Vanishing}

\indent
In this section we prove solvability of the $\bar\partial$-equation on affine smooth
algebraic curves. Let ${\cal X}$ be a nonsingular algebraic curve,
and let $\varrho:{\cal X}\to\C\P^2$
be an immersion of ${\cal X}$ as in \eqref{varrhoMap} such that
\begin{equation}\label{CurveC}
{\cal C}=\varrho\left({\cal X}\right)=\left\{z\in\C\P^2: P(z)=0\right\}
\end{equation}
is a plane curve of $\deg P=d$ with at most nodes as singularities.
Without loss of generality we may assume that the intersection of ${\cal C}$ with the line at infinity
\begin{equation}\label{InfiniteLine}
{\cal C}\cap\left\{(z_0,z_1,z_2)\in \C\P^2: z_0=0\right\}=\left\{z^{(1)},\dots,z^{(r)}\right\}
\end{equation}
consists of $r$ points with multiplicities $m_1,\dots,m_r$ such that $\sum_{i=1}^r m_i=d$.
We denote
\begin{equation*}
\mathring{\cal C}={\cal C}\setminus\left\{z^{(1)},\dots,z^{(r)}\right\},
\end{equation*}
and
\begin{equation*}
\mathring{\cal X}={\cal X}
\setminus\left\{\varrho^{-1}(z^{(1)}),\dots,\varrho^{-1}(z^{(r)})\right\}.
\end{equation*}

\begin{proposition}\label{Solvability} (Hodge-Kodaira Vanishing Theorem for affine curves)
Let ${\cal X}$ and ${\cal C}\subset \C\P^2$
be curves as in \eqref{CurveC}, with ${\cal C}$ of degree $d$ satisfying
condition \eqref{InfiniteLine}. Let $\phi^{(0,1)}\in{\cal E}(\mathring{\cal X})$
be a form on $\mathring{\cal X}$, and $\phi_*$ be its direct image on $\mathring{\cal C}$.
If for some $\ell$ satisfying $\ell>d-3$ the form $\zeta_0^{\ell}\phi_*(\zeta)$
admits an extension $\psi(\zeta)\in C^{(0,1)}({\cal C})$, then there exists a function
$g\in{\cal E}(\mathring{\cal X})$ such that
\begin{equation}\label{AffineSolution}
\left\{\ \begin{aligned}
&\bar\partial g=\phi,\vspace{0.1in}\\
&|g(\zeta)|\leq C|\zeta_0|^{-\ell}.
\end{aligned}\right.
\end{equation}
\end{proposition}
\begin{proof}
Let $z^{(j)}\in {\cal C}\cap\left\{\C\P^2: z_0=0\right\}$ be one of the points of ${\cal C}$
at infinity. We consider a neighborhood $U^{(j)}$ of $z^{(j)}$ in $\C^2$ with
coordinates $\zeta_1,\zeta_2$ such that
\begin{equation*}
{\cal C}\cap U^{(j)}=\left\{\zeta\in U^{(j)}: \zeta_2=0\right\}.
\end{equation*}
Using Cauchy-Green formula we solve the $\bar\partial$-equation on ${\cal C}\cap U^{(j)}$
and obtain a function
\begin{equation*}
g^{(j)}(\zeta_1)\in C({\cal C})\cap C^{\infty}(\mathring{\cal C})
\end{equation*}
with compact support in ${\cal C}\cap U^{(j)}$ satisfying equality
\begin{equation*}
\bar\partial g^{(j)}=\psi\big|_{{\cal C}\cap V^{(j)}},
\end{equation*}
where $V^{(j)}\Subset U^{(j)}$.\\
\indent
Then the form
\begin{equation*}
\tilde{\psi}=\psi-\sum_{j=1}^d\bar\partial g^{(j)}
\end{equation*}
defines a $\bar\partial$-closed residual current on ${\cal C}$ of degree $\ell>d-3$,
and therefore, using Corollary~\ref{PositiveHomogeneity} we obtain the existence
of a function
$\tilde{g}=I_1[\zeta_0^{\ell}\cdot \tilde{\psi}(\zeta)]\in C^{(0,0)}(\mathring{\mathcal C})$
satisfying the equation
\begin{equation*}
\bar\partial\tilde{g}=\tilde{\psi}.
\end{equation*}
\indent
Therefore function $g_*=\tilde{g}+\sum_{j=1}^d g^{(j)}$ satisfies the equation
\begin{equation*}
\bar\partial g_*(\zeta)=\psi(\zeta)=\zeta_0^{\ell}\phi_*(\zeta),
\end{equation*}
and the function $g(\zeta)=(\zeta_0^{-\ell}g_*)^*$ satisfies conditions \eqref{AffineSolution}.
\end{proof}
\indent
Combining results of Theorem~\ref{ExplicitHodge} and Proposition~\ref{Solvability} we obtain
the following proposition.

\begin{proposition}\label{CompactSolvability}
Let ${\cal X}$ and ${\cal C}\subset \C\P^2$
be curves as in \eqref{CurveC}, with ${\cal C}$ satisfying condition \eqref{InfiniteLine}.
Let $\phi^{(0,1)}\in{\cal E}(\mathring{\cal X})$
be a form on $\mathring{\cal X}$, such that its direct image
$\phi_*$ has a compact support in $\mathring{\cal C}$. Then for arbitrary $\ell>d-3$
the function
\begin{equation*}
g(\zeta)=\left(\zeta_0^{-\ell}\cdot I_1[\zeta_0^{\ell}\phi_*
-H_1[\zeta_0^{\ell}\phi_*]]\right)^*
\end{equation*}
satisfies equation
\begin{equation*}
\bar\partial g=\phi.
\end{equation*}
\end{proposition}

\end{document}